\documentclass[11pt, reqno]{amsart}

\usepackage{setspace}
\usepackage{color}
\usepackage{hyperref}  
\usepackage{graphicx}
\usepackage{datetime}
\usepackage{verbatim}
\usepackage{srcltx}

\usepackage{parskip}
\usepackage{amsfonts}
\usepackage{amsmath}

\def\EE{\mathbb E}

\def\({\left(}
\def\){\right)}  
\def\<{\langle}
\def\>{\rangle}
\let\:=\colon
\let\epsilon\varepsilon
\let\log\log

\newcommand{\hm}[1]{\leavevmode{\marginpar{\tiny%
$\hbox to 0mm{\hspace*{-0.5mm}$\leftarrow$\hss}%
\vcenter{\vrule depth 0.1mm height 0.1mm width \the\marginparwidth}%
\hbox to 0mm{\hss$\rightarrow$\hspace*{-0.5mm}}$\\\relax\raggedright #1}}}

\def\({\left(}
\def\){\right)}
\def\:{\colon}
\def\[{\left[}
\def\]{\right]}

\let\epsilon\varepsilon

\DeclareMathAlphabet{\Sss}{U}{bbmss}{m}{n}

\usepackage{amssymb}
\newcommand{\ignore}[1]{{}}



\newtheoremstyle{note}
  {4pt}
  {4pt}
  {\sl}
  {}
  {\bfseries}
  {.}
  {.5em}
  {}

\newtheoremstyle{introthms}
  {3pt}
  {3pt}
  {\normalfont}
  {}
  {\bfseries}
  {.}
  {.5em}
  {\thmnote{#3}}

\newtheoremstyle{cases}
  {2pt}
  {2pt}
  {\rm}
  {}
  {\bfseries}
  {.}
  {.3em}
  {}

\theoremstyle{plain}
\newtheorem{theorem}              {Theorem}       
\newtheorem{lemma}      [theorem] {Lemma}         
\newtheorem{corollary}  [theorem] {Corollary}     
\newtheorem{fact}       [theorem] {Fact}          

\theoremstyle{cases}

\theoremstyle{introthms}

\theoremstyle{note}

\newtheorem*{acknowledge}{\textbf{Acknowledgement}}

\title{On constant-multiple-free sets \\ contained in a random set of integers}

\author[S.~J.~Lee]{Sang June Lee}
\address{School of Mathematics, Korea Institute for Advanced Study (KIAS), 
Seoul 130-722, South Korea}
\email{sjlee242@gmail.com}
\thanks{The author was supported by Korea Institute for Advanced
  Study (KIAS) grant funded by the Korea government (MEST)}

\date{\today, \currenttime}

\begin{document}
\pagestyle{plain}
\thispagestyle{empty}
\footskip=30pt

\shortdate
\settimeformat{ampmtime}


\begin{abstract}   For a rational number $r>1$, a set $A$ of positive integers is called an \textit{$r$-multiple-free set} if $A$ does not contain any solution of the equation $rx = y$. The extremal problem on estimating the maximum possible size of $r$-multiple-free sets contained in $[n]:=\{1,2,\cdots,n\}$ has been studied for its own interest in combinatorial number theory and for application to coding theory.  Let $a$, $b$ be positive integers such that $a<b$ and the greatest common divisor of $a$ and $b$ is $1$. Wakeham and Wood showed that the maximum size of $(b/a)$-multiple-free sets contained in $[n]$ is $\frac{b}{b+1}n+O(\log n)$.
 
 In this note we generalize this result as follows. For a real number $p\in (0,1)$, let $[n]_p$ be a  set of integers obtained by choosing each element $i\in [n]$ randomly and independently with probability $p$. We
show that the maximum possible size of $(b/a)$-multiple-free sets contained in $[n]_p$ is $\frac{b}{b+p}pn+O(\sqrt{pn}\log n \log \log n)$ with probability that goes to $1$ as $n\to \infty$.
\end{abstract}

\maketitle

\section{Introduction}

In recent years a trend in extremal combinatorics concerned with investigating how classical extremal results in \textit{dense} environments transfer to \textit{sparse} settings, and it has seen to be a fruitful subject of research. Especially, in combinatorial number theory, the following extremal problem in a dense environment  has been well-studied and successively extended to sparse settings: Fix an equation and estimate the maximum size of subsets of $[n]:=\{1, 2, \cdots, n\}$ containing no non-trivial solutions of the given equation.

An example of this line of research is a version of Roth's theorem~\cite{roth53} on arithmetic progressions of length $3$ (with respect to the equation $x_1+x_3=2x_2$) for random subsets of
integers in Kohayakawa--\L uczak--R\"odl~\cite{KLRap3}.  Also, Szemer\'edi's theorem~\cite{szemeredi75} was transfered to random subsets of integers in Conlon--Gowers~\cite{conlon:_theor_in_spars_random} and
Schacht~\cite{schacht:_extrem_resul}. 
The result of Erd\H{o}s--Tur\'{a}n~\cite{Erdos41}, Chowla~\cite{Chowla44}, and
Erd\H{o}s~\cite{Erdos44} in 1940s on the maximum size of Sidon sets in $[n]$ was extended in~\cite{Sidon_SODA, Sidon} to sparse random subsets of $[n]$, where a \textit{Sidon set} is a set of positive integers not containing any non-trivial solution of $x_1+x_2=y_1+y_2$.

In this paper we transfer the following extremal results to sparse random subsets. For a rational number $r>1$, a set $A$ of positive integers is called an \textit{$r$-multiple-free set} if $A$ does not contain any solution of $rx = y$. 
 An interesting problem on $r$-multiple-free sets is of estimating the maximum possible size $f_r(n)$ of $r$-multiple-free sets contained in $[n]:=\{1,2,\cdots,n\}$. This extremal problem has been studied in~\cite{Wang1989, Leung1994, Wakeham} for its own interest in combinatorial number theory, and also was applied to coding theory in~\cite{Jimbo2007}.
 
Wang~\cite{Wang1989} showed that $f_2(n)=\frac{2}{3}n+O(\log n)$. Leung and Wei~\cite{Leung1994} proved that for every integer $r>1$, $f_r(n)=\frac{r}{r+1}n+O(\log n)$.
Wakeham and Wood~\cite{Wakeham} extended it to rational numbers as follows. For positive integers $a$ and $b$, let $gcd(b,a)$ be the greatest common divisor of $a$ and $b$.
\begin{theorem}[Wakeham and Wood~\cite{Wakeham}]\label{thm:dense} Let $a,b$ be positive integers with $a<b$ and $gcd(b,a)=1$. Then
  $$f_{b/a}(n)=\frac{b}{b+1}n+O\(\log n\).$$
\end{theorem}

We shall investigate the maximum size of constant-multiple-free sets contained in a random subset of $[n]$. Let $[n]_p$ be a random subset of $[n]$ obtained by choosing each element in $[n]$ independently with probability $p$. 
Let $f_r([n]_p)$ denote the maximum size of $r$-multiple-free sets contained in $[n]_p$.
We are interested in the behavior of $f_r([n]_p)$ for every rational number $r>1$.

Theorem~\ref{thm:dense} gives the answer of the above question for the case $p=1$. On the other hand, if $p=o(1)$, then the usual deletion methods give that \textit{with high probability} (that is, with probability that goes to 1 as $n\to \infty$) the maximum size of $(b/a)$-multiple-free sets contained in $[n]_p$ is $np(1-o(1))$. Hence, from now on, we consider $p$ as a real number with $0<p<1$.

Using Chernoff bounds (for example, see Lemma~\ref{lem:Chernoff}), Theorem~\ref{thm:dense} easily implies the following:
\begin{fact}\label{fact:lower}Let $p\in (0,1)$ and let $a, b$ be positive integers such that $a<b$ and $gcd(a,b)=1$. Let $\omega$ be a function of $n$ that goes to $\infty$ arbitrarily slowly as $n\to \infty$. With high probability, there is a $(b/a)$-multiple-free set in $[n]_p$ of size
$$ \frac{b}{b+1}pn+\omega\sqrt{pn}.$$
\end{fact}

Fact~\ref{fact:lower} gives a lower bound on $f_{b/a}([n]_p)$ that is off from the right value of $f_{b/a}([n]_p)$.  The main result of this paper is the following:

\begin{theorem}\label{thm:main}Let $p\in (0,1)$ and let $a, b$ be positive integers such that $a<b$ and $gcd(a,b)=1$. Then, with high probability, $$f_{b/a}\([n]_p\)=\frac{b}{b+p}pn+O\(\sqrt{pn}\log n\log\log n\).$$
\end{theorem}

\begin{figure}
\begin{center}
\includegraphics[scale=0.15]{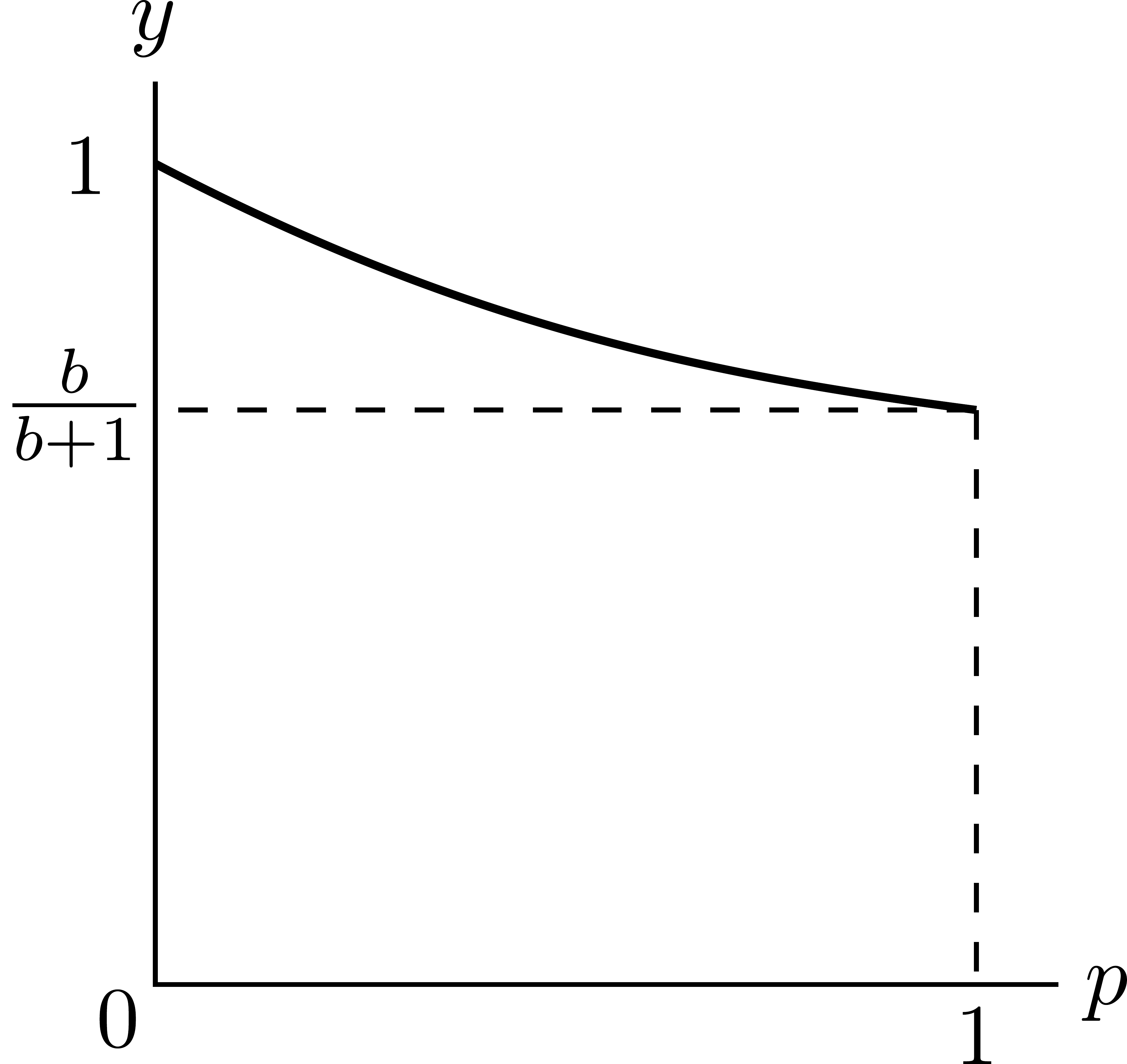}
\end{center}
\caption{The graph of $y=b/(b+p)$ for $0\leq p\leq 1$}
\label{fig:p_function}
\end{figure}

The ratio $ {f_{b/a}([n]_p) \over np}$ goes from $1$ to $b\over b+1$ as
p varies from $0$ to $1$ (See Figure~\ref{fig:p_function}).
The proof of Theorem~\ref{thm:main} is given in Sections~\ref{sec:main theorem} and~\ref{sec:main lemma} by using a graph theoretic method.

\section{Proof of Theorem~\ref{thm:main}}
\label{sec:main theorem}

In order to show Theorem~\ref{thm:main}, we use a graph theoretic approach that was used in Wakeham and Wood~\cite{Wakeham}. Let $r=b/a>1$ be a rational number. 
Let $D=(V,E)$ be the directed graph with the vertex set $V=[n]$ in which the set $E$ of arcs (or directed edges) is $\{(x,y) :\; rx=y\}.$
Let $D[[n]_p]$ be the subgraph of $D$ induced on $[n]_p$.
Observe that $f_r([n]_p)$ is the same as the independence number $\alpha(D[[n]_p])$ of $D[[n]_p]$. 

We consider structures of $D[[n]_p]$. The indegree and outdegree of each vertex in $D$ are at most $1$. Also, there is no directed cycle in $D$ because $(x,y)\in E$ implies $x<y$. Therefore, each component of $D$ or $D[[n]_p]$ is a directed path.

In order to obtain an independent set of $D[[n]_p]$ of maximum size, we consider such an independent set componentwise. Let $C$ be a component of $D[[n]_p]$. As we mentioned above, $C$ is a directed path. Let $V(C)=\{u_0, u_1, u_2, \cdots, u_i, \cdots, u_l\}$ be the vertex set of $C$ such that $u_j<u_{j+1}$ and $(u_j, u_{j+1})\in E$ for $0\leq j \leq l-1$.
Observe that $V^*(C):=\{u_0, u_2, u_4, \cdots\}\subset V(C)$ forms an independent set of $C$ of maximum size. 
Therefore, the set $$\displaystyle T^*:=\bigcup_{C}V^*(C),$$ where $C$ is each component of $D[[n]_p]$, forms an independent set of $D[[n]_p]$ of maximum size. Hence, we have the following.
\begin{lemma}
$f_r([n]_p)=|T^*|.$
\end{lemma}

Thus, in order to show Theorem~\ref{thm:main}, it suffices to show the following.
\begin{lemma}\label{lem:main}Let $p\in (0,1)$ and let $a, b$ be natural numbers such that $a<b$ and $gcd(a,b)=1$. Then, with high probability, $$|T^*|=\frac{b}{b+p}pn+O\(\sqrt{pn}\log n\log\log n\).$$
\end{lemma}
The proof of Lemma~\ref{lem:main} is given in Section~\ref{sec:main lemma}.

\section{Proof of Lemma~\ref{lem:main}}\label{sec:main lemma}

From now on, we show Lemma~\ref{lem:main}.
For positive integers $b$ and $k$, let $k$ be an $i$-th subpower of $b$ if $k=b^i l$ for some $l\not\equiv 0$ (mod $b$). Let $T_i$ be the set of $i$-th subpowers of $b$ in $[n]$.
Let $T^*_i\subset T_i$ denote the set of $i$-th subpowers $v$ of $b$ in $[n]_p$ such that $v$ is at an even distance from the smallest vertex of the component of $D[[n]_p]$ containing $v$. Observe that $T^*=\bigsqcup_i T^*_i,$ and hence, 
\begin{equation}
\label{eq:T^*}
|T^*|=\sum_i |T^*_i|.
\end{equation}

In Section~\ref{ssec:expectation}, we estimate the expected value $\EE(|T^*|)$. Section~\ref{ssec:concentration} deals with a concentration result of $|T^*|$ with high probability.

\subsection{Expectation}\label{ssec:expectation}
We first estimate $\EE\(|T^*_i|\)$ and their sum $\EE\(|T^*|\)$.
Recall that $T_i$ denotes the set of $i$-th subpowers of $b$ in $[n]$. Note that since $1\leq b^i\leq n$, the range of $i$ is $0\leq i\leq \log_b n$. It is clear that
$$T_i=\Big\{b^ix \;\big|\; 1\leq x\leq \frac{n}{b^i}, \quad x\not\equiv 0 \mbox{ (mod $b$) }\Big\}.$$
Hence we have the following:
\begin{fact}\label{fact:T_i}
\begin{equation}\label{eq:T_i}
|T_i|=\frac{b-1}{b}\frac{n}{b^i}\pm 1.
\end{equation}
\end{fact}

 We consider two cases separately, based on the parity of $i$.

\begin{lemma}\label{lem:mean_even} For $0\leq j\leq (\log_b n)/2$, we have 
\begin{equation*}
\EE\(|T^*_{2j}|\)=\frac{b-1}{b(1+p)}pn\(\frac{1}{b^{2j}}+\(\frac{p}{b}\)^{2j}p\)\pm 1.
\end{equation*}
\end{lemma}

\begin{proof} First we consider $\Pr \[v\in T_{2j} \mbox{ is in } T^*_{2j}\]$.
Let $\{v_0, v_1, v_2, \cdots \}$, where $v_i<v_{i+1}$, be the vertex set of the component of $D$ containing $v$. Observe that $v_i\in T_{i}$, and hence, $v=v_{2j}$.
The event that $v\in T_{2j}$ is in $T^*_{2j}$ happens only when one of the following holds:
\begin{itemize}
\item There is some $r$ with $0\leq r\leq j-1$ such that $v_{2j-1-2r}\not\in [n]_p$ and $v_i\in [n]_p$ for all $ 2j-2r \leq i\leq 2j$.
\item The vertices $v_0,v_1,\cdots, v_{2j}$ are in $[n]_p$.
\end{itemize}
Hence, we have 
\begin{equation}\label{eq:prob}
\Pr \[v\in T_{2j} \mbox{ is in } T^*_{2j}\]=p\((1-p)+p^2(1-p)+\cdots +p^{2j-2}(1-p)+p^{2j}\).
\end{equation}

Thus we infer
\begin{eqnarray*}
\EE\(|T^*_{2j}|\)&=&|T_{2j}|\cdot\Pr \[v\in T_{2j} \mbox{ is in } T^*_{2j}\] \\ 
&\overset{\eqref{eq:T_i}, \eqref{eq:prob}}{=}&\(\frac{b-1}{b}\frac{n}{b^{2j}}\pm 1\)p\((1-p)\frac{1-p^{2j}}{1-p^2}+p^{2j}\) \\
 &=&\frac{b-1}{b(1+p)}pn\(\frac{1}{b^{2j}}+\frac{p^{2j}}{b^{2j}}p\)\pm 1,
\end{eqnarray*}
which completes the proof of Lemma~\ref{lem:mean_even}.
\end{proof}

\begin{lemma}\label{lem:mean_odd} For $1\leq j\leq (\log_b n)/2$, we have 
\begin{equation*}
\EE\(|T^*_{2j-1}|\) = \frac{b-1}{b(1+p)}pn\(\frac{1}{b^{2j-1}}-\(\frac{p}{b}\)^{2j-1}p\) \pm 1.
\end{equation*}
\end{lemma}

\begin{proof} Using an argument similar to the proof of~\eqref{eq:prob}, one may obtain that
\begin{equation}\label{eq:prob(2)}
\Pr\[v\in T_{2j-1} \mbox{ is in } T^*_{2j-1}\]=p\((1-p)+p^2(1-p)+\cdots +p^{2j-2}(1-p)\).
\end{equation}

Thus we infer
\begin{eqnarray*}
\EE\(|T^*_{2j-1}|\)&=& |T_{2j-1}|\cdot \Pr\[v\in T_{2j-1} \mbox{ is in } T^*_{2j-1}\] \\
&\overset{\eqref{eq:T_i}, \eqref{eq:prob(2)}}{=}& \((b-1)\frac{n}{b^{2j}}\pm 1\) p(1-p)\frac{1-p^{2j}}{1-p^2} \\ 
&=&\frac{b-1}{1+p}pn\(\frac{1}{b^{2j}}-\(\frac{p}{b}\)^{2j}\)\pm 1,
\end{eqnarray*}
which completes the proof of Lemma~\ref{lem:mean_odd}.
\end{proof}

Lemmas~\ref{lem:mean_even} and~\ref{lem:mean_odd} immediately imply the following.

\begin{corollary}\label{coro:mean}For $0\leq i\leq \log_b n$, we have 
\begin{equation}\label{eq:mean}
\EE\(|T^*_i|\) = \frac{b-1}{b(1+p)}pn\(\frac{1}{b^i}+\(-\frac{p}{b}\)^i p\) \pm 1.
\end{equation}
\end{corollary}

Summing over all $i$ with $0\leq i\leq \log_b n$, we have the following.

\begin{corollary}\label{clm:mean_total}
\begin{equation*}
\EE\(|T^*|\)=\sum_{i=0}^{\log_b n} \EE(|T^*_i|)=\frac{b}{b+p}pn +O(\log n).
\end{equation*}
\end{corollary}

\begin{proof}

One may easily see that for $|x|\geq b\geq 2$,
\begin{equation}\label{eq:sum_geom} \sum_{i=0}^{\log_b n} \frac{1}{x^i} =\frac{x}{x-1}+O\(\frac{1}{n}\). \end{equation} 
Corollary~\ref{coro:mean} yields that for $b\geq 2$
\begin{eqnarray}\label{eq:even_sum}
\sum_{i=0}^{\log_b n} \EE(|T^*_{j}|) 
&\overset{\eqref{eq:mean}}{=}& \sum_{i=0}^{\log_b n} \[\frac{b-1}{b(1+p)}pn\(\frac{1}{b^{i}}+\(-\frac{p}{b}\)^ip\)\pm 1\] \nonumber \\
&\overset{\eqref{eq:sum_geom}}{=}& \frac{b-1}{b(1+p)}pn\[\frac{b}{b-1}+O\(\frac{1}{n}\)+\frac{-b/p}{-b/p-1}p+O\(\frac{1}{n}\)\]+O(\log n) \nonumber \\
&=&\frac{b}{b+p}pn+O(\log n),
\end{eqnarray}
which completes the proof of Corollary~\ref{clm:mean_total}.
\end{proof}

\subsection{Concentration}\label{ssec:concentration}
Next we consider a concentration result of $|T^*_{i}|$. In other words, we show that $|T^*_{i}|$ is around its expectation with high probability.  We will apply the following version of Chernoff bounds.

\begin{lemma}[Chernoff bound]\label{lem:Chernoff}
Let $X_i$ be independent random variables such that $\Pr[X_i=1]=p_i$ and $\Pr[X_i=0]=1-p_i$, and let $X=\sum_{i=1}^{n} X_i$. Then for any $\lambda\geq 0$,
\begin{eqnarray}
\Pr\[ X\geq (1+\lambda)\EE(X)\] & \leq  & e^{-\frac{\lambda^2}{2+\lambda}\EE(X)}, \label{eq:Chernoff_upper} \\
\Pr\[ X\leq (1-\lambda)\EE(X)\]  & \leq & e^{-\frac{\lambda^2}{2}\EE(X)}. \label{eq:Chernoff_lower} 
\end{eqnarray}

In particular, for $0\leq \lambda\leq 1$,
\begin{equation}\label{eq:Chernoff}\Pr \[|X-\EE(X)|\geq \lambda\EE(X)\]\leq 2 e^{-\frac{\lambda^2}{3}\EE(X)}. \end{equation}

\end{lemma}

We first consider the case when $0\leq i\leq 0.9\log_b n$.

\begin{lemma}\label{lem:concentration_large_i} For $0\leq i\leq 0.9\log_b n$, we have
\begin{equation}\label{eq:T^*_i}
|T^*_i|=\EE\(|T^*_i|\)+O\(\sqrt{pn}\log\log n\)
\end{equation}
with probability at least $1-2e^{-\frac{1}{3}(\log\log n)^2}$.

\end{lemma}

\begin{proof}
Fix $i$. 
If $k\in T_i \subset [n]$, then let $$X_k=\left\{\begin{array}{rl} 1 & \mbox{with probability } p^*  \\ 0 & \mbox{with probability } 1-p^*.
\end{array} \right.$$
where $p^*=\Pr \[v\in T_{i} \mbox{ is in } T^*_{i}\]$.
Otherwise, let $X_k=0$ with probability $1$. Let $X=\sum_{k=1}^n X_k$.
Observe that
\begin{equation}\label{eq:X}X=|T^*_i|\end{equation}
as random variables.

Note that for each $k\in T_i$, the event that $k\in T^*_i$ depends only on the events that $v\in [n]_p$, where the vertices $v$ are in the component of $D$ containing $k$ and $v\leq k$. Hence, $X_k$ are independent for all $k\in T_i$. Therefore we are able to use Chernoff bounds (Lemma~\ref{lem:Chernoff}) for a concentration result of $X$. 

Set $\displaystyle \lambda=\frac{\log\log n}{\sqrt{\EE(X)}}.$  Note that $0\leq \lambda\leq 1$ for $0\leq i\leq 0.9\log_b n$ since
$$\EE(X)\geq \Omega\(pn\frac{\epsilon_p}{b^i}\)\geq \Omega\(pn\frac{\epsilon_p}{n^{0.9}}\)=\Omega\(\epsilon_p pn^{0.1}\),
$$
where $\epsilon_p$ is a positive constant such that $\epsilon_p\to 0$ as $p\to 1$.
The inequality~\eqref{eq:Chernoff} yields that 
\begin{eqnarray}\label{eq:T^*_i(1)}
\Pr\[|X-\EE(X)|\geq  \sqrt{\EE\(X\)}\log\log n\]\leq 2e^{-\frac{1}{3}(\log\log n)^2}.
\end{eqnarray}
Corollary~\ref{coro:mean} yields that $ \EE\(|X|\)= O(pn),$ and hence, we infer that
\begin{equation*}
X=\EE\(X\)+O\(\sqrt{pn}\log\log n\)
\end{equation*}
with probability at least $1-2e^{-\frac{1}{3}(\log\log n)^2}$.
This together with~\eqref{eq:X} completes the proof of Lemma~\ref{lem:concentration_large_i}.
\end{proof}

Next we consider the remaining case when $0.9\log_b n\leq i\leq \log_b n$.

\begin{lemma}\label{lem:concentration_small_i} For $0.9\log_b n\leq i\leq \log_b n$, we have
$|T^*_i|=O\((pn)^{0.1}\),$
with probability at least $1-e^{-(\log\log n)^2}$.
\end{lemma}

\begin{proof}We define a random variable $X$ as in~\eqref{eq:X}, that is, $X=|T^*_i|$. 
Set $\displaystyle\lambda=\frac{2(\log\log n)^2}{\EE(X)}.$  The inequality~\eqref{eq:Chernoff_upper} yields that 
\begin{eqnarray*}
\Pr\[X\geq (1+\lambda)\EE(X)\]\leq e^{-\frac{\lambda}{2}\EE\(X\)}=e^{-(\log\log n)^2},
\end{eqnarray*}
and hence,
\begin{eqnarray}\label{eq:T^*_i(2)}
\Pr\[X\geq \EE(X)+2(\log\log n)^2\]\leq e^{-(\log\log n)^2}.
\end{eqnarray}
In other words, 
\begin{equation}\label{eq:X(200)}
X\leq \EE(X)+2(\log\log n)^2
\end{equation}
with probability at least $1-e^{-(\log\log n)^2}.$

Corollary~\ref{coro:mean} gives that
\begin{equation}\label{eq:X(100)}
 \EE(X)=O\(pn\frac{1}{b^i}\)=O\(p n^{0.1}\)=O\((pn)^{0.1}\),
\end{equation}
where the second inequality holds for $i\geq 0.9\log_b n$.
Thus, combining~\eqref{eq:X(200)} and~\eqref{eq:X(100)} completes the proof of Lemma~\ref{lem:concentration_small_i}.
\end{proof}

Now we are ready to show Lemma~\ref{lem:main}.

\begin{proof}[Proof of Lemma~\ref{lem:main}] We have that
\begin{eqnarray*}
|T^*|&=&\sum_{i=1}^{\log_b n} |T^*_i| =\sum_{i=1}^{\lfloor 0.9\log_b n\rfloor} |T^*_i|+\sum_{i=\lfloor 0.9\log_b n\rfloor+1}^{\log_b n} |T^*_i|.
\end{eqnarray*}
Lemmas~\ref{lem:concentration_large_i} and~\ref{lem:concentration_small_i} give that
\begin{eqnarray*}
|T^*|&= & \sum_{i=1}^{\log_b n} \EE\(|T^*_i|\)+O\(\sqrt{pn}\log n\log\log n\),
\end{eqnarray*}
with probability at least
\begin{equation}
1-(\log_b n)\cdot 2e^{-\frac{1}{3}(\log\log n)^2}=1-2e^{\log\log_b n-\frac{1}{3}(\log\log n)^2}=1-o(1).
\end{equation}

This together with Corollary~\ref{clm:mean_total} implies that with high probability
\begin{eqnarray*}
|T^*|&= &\frac{b}{b+p}pn +O\(\sqrt{pn}\log n\log\log n\),
\end{eqnarray*}
which completes the proof of Lemma~\ref{lem:main}.
\end{proof}

\begin{acknowledge}
  The author thanks
 Yoshiharu Kohayakawa for his helpful comments and suggestions, and thanks Jaigyoung Choe for his support at Korea Institute for Advanced Study.
\end{acknowledge}

\end{document}